\documentclass[reqno]{amsart}
\usepackage{amssymb}

\newcounter{makeconstant}
\newenvironment{makeconstant}%
{\refstepcounter{makeconstant}}%
{}
\def\mc#1{\begin{makeconstant}\label{#1}\end{makeconstant}}

\def\({\left(}\def\){\right)}
\def\imag{{\rm i}}
\def\eul{{\rm e}}
\def\dee{\thinspace{\rm{d}}}

\DeclareMathOperator{\genlin}{GL}

\newcommand{\divides}{|}

\newtheorem{theorem}{Theorem}
\newtheorem{lemma}[theorem]{Lemma}

\theoremstyle{definition}
\newtheorem{example}[theorem]{Example}

\renewcommand{\le}{\leqslant}
\renewcommand{\ge}{\geqslant}
\renewcommand{\epsilon}{\varepsilon}

\def\bigo{\operatorname{O}}    
\def\littleo{\operatorname{o}} 

\begin{document}

\title{Mertens' theorem for toral automorphisms}
\author{Sawian Jaidee}\address{(SJ) Department of Mathematics,
123 Mittraphab Road, Khon Kaen University 40002,Thailand}
\author{Shaun Stevens}
\author{Thomas Ward}\address{(SS \& TW) School of Mathematics, University of East Anglia, Norwich
NR4 7TJ, UK}
\date{\today}

\thanks{}

\begin{abstract}
A dynamical Mertens' theorem for ergodic toral automorphisms
with error term~$\bigo(N^{-1})$ is found, and the influence of
resonances among the eigenvalues of unit modulus is examined.
Examples are found with many more, and with many fewer,
periodic orbits than expected.
\end{abstract}

\maketitle

\section{Introduction}


Discrete dynamical analogs of Mertens' theorem concern a
map~$T:X\to X$, and are motivated by work of
Sharp~\cite{MR1139566} on Axiom A flows. A set of the form
\[
\tau=\{x,T(x),\dots,T^k(x)=x\}
\]
with cardinality~$k$ is called a closed orbit of
length~$\vert\tau\vert=k$, and the results provide asymptotics
for a weighted sum over closed orbits. For the discrete case of
a hyperbolic diffeomorphism~$T$, we always have
\[
M_T(N):=\sum_{\vert\tau\vert\le
N}\frac{1}{\eul^{h\vert\tau\vert}}\sim\log(N),
\]
where~$h$ is the topological entropy, with more explicit
additional terms in many cases. The main term~$\log(N)$ is not
really related to the dynamical system, but is a consequence of
the fact that the number of orbits of length~$n$
is~$\frac{1}{n}\eul^{hn}+\bigo(\eul^{h'\!n})$ for some~$h'<h$
(see~\cite{apisit}). Without the assumption of hyperbolicity,
the asymptotics change significantly, and in particular depend
on the dynamical system. For quasihyperbolic (ergodic but not
hyperbolic) toral automorphisms, Noorani~\cite{MR1787646} finds
an analogue of Mertens' theorem in\mc{nooraniconstant} the form
\begin{equation}\label{nooranimain}
M_T(N)=m\log(N)+C_{\ref{nooraniconstant}}+\littleo(1)
\end{equation}
for some~$m\in\mathbb N$. The
constant~$C_{\ref{nooraniconstant}}$ is related to analytic
data coming from the dynamical zeta function. For more general
non-hyperbolic group automorphisms, the coefficient of the main
term may be non-integral (see~\cite{MR2339472} for example).

In this note Noorani's result~\eqref{nooranimain} with improved error
term~$\bigo(N^{-1})$ is recovered using elementary arguments,
and the coefficient~$m$ of the main term in~\eqref{nooranimain}
is expressed as an integral over a sub-torus. This reveals the
effect of resonances between the eigenvalues of unit modulus,
and examples show that the value of~$m$ may be very different to
the generic value given in~\cite{MR1787646}.

\section{Toral automorphisms}

Let~$T:\mathbb T^d\to\mathbb T^d$ be a toral automorphism
corresponding to a matrix~$A_T$ in~$\genlin_{d}(\mathbb Z)$
with eigenvalues~$\{\lambda_i\mid1\le i\le d\}$, arranged so
that
\[
\vert\lambda_1\vert\ge\cdots\ge\vert\lambda_s\vert>1=
\vert\lambda_{s+1}\vert=\cdots=\vert\lambda_{s+2t}\vert>
\vert\lambda_{s+2t+1}\vert\ge\cdots\ge\vert\lambda_d\vert.
\]
The map~$T$ is \emph{ergodic} with respect to Lebesgue measure
if no eigenvalue is a root of unity, is \emph{hyperbolic} if in
addition~$t=0$ (that is, there are no eigenvalues of unit
modulus), and is \emph{quasihyperbolic} if it is ergodic
and~$t>0$. The topological entropy of~$T$ is given
by~$h=h(T)=\sum_{j=1}^{s}\log\vert\lambda_j\vert$.

\begin{theorem}\label{theorem}
Let~$T$ be a quasihyperbolic
toral automorphism with topological
entropy~$h$. Then there are\mc{mainconstantquasicase}
constants~$C_{\ref{mainconstantquasicase}}$ and~$m\ge1$ with
\[
\sum_{\vert\tau\vert\le N}\frac{1}{\eul^{h\vert\tau\vert}}=m\log
N+C_{\ref{mainconstantquasicase}}+\bigo\(N^{-1}\).
\]
The coefficient~$m$ in the main term is given by
\[
m=\int_X\prod_{i=1}^{t}
\(2-2\cos(2\pi x_i)\)\dee x_1\dots\dee x_{t},
\]
where~$X\subset\mathbb T^d$ is the closure
of~$\{(n\theta_1,\dots,n\theta_{t})\mid n\in\mathbb Z\}$,
and~$\eul^{\pm2\pi{\imag}\theta_1},\dots,
\eul^{\pm2\pi{\imag}\theta_{t}}$ are the eigenvalues with unit
modulus of the matrix defining~$T$.
\end{theorem}

As we will see in Example~\ref{examples}, the quantity~$m$
appearing in Theorem~\ref{theorem} takes on a wide range of
values. In particular,~$m$ may be much larger, or much smaller,
than its generic value~$2^t$.

\begin{proof}
Since~$T$ is ergodic,
\begin{equation*}
F_T(n)=\vert\{x\in\mathbb T^d\mid T^n(x)=x\}\vert=\vert\mathbb
Z^d/(A_T^n-I)\mathbb Z^d\vert=\prod_{i=1}^{d}\vert\lambda_i^n-1\vert,
\end{equation*}
so
\[
O_T(n)=\frac{1}{n}\sum_{m\divides
n}\mu(n/m)\prod_{i=1}^{d}\vert\lambda_i^m-1\vert.
\]
Write~$\Lambda=\prod_{i=1}^{s}\lambda_i$ (so the topological entropy of~$T$
is~$\log\vert\Lambda\vert$) and
\[
\kappa=\min\{\vert\lambda_s\vert,\vert\lambda_{s+2t+1}\vert^{-1}\}>1.
\]
The eigenvalues of unit modulus contribute nothing to the
topological entropy, but multiply the
approximation~$\vert\Lambda\vert^n$ to~$F_T(n)$ by an
almost-periodic factor bounded above by~$2^{2t}$ and bounded
below by~$A/n^B$ for some~$A,B>0$, by Baker's theorem
(see~\cite[Ch.~3]{MR1700272} for this argument).

\begin{lemma}\label{lemma:1}
$\left\vert
F_T(n)-\vert\Lambda\vert^n\displaystyle\prod_{i=s+1}^{s+2t}\vert\lambda_i^n-1\vert
\right\vert\cdot\vert\Lambda\vert^{-n}=\bigo(\kappa^{-n}).$
\end{lemma}

\begin{proof}
We have
\begin{equation}\label{equation:definesABC}
\prod_{i=1}^{d}(\lambda_i^n-1)=
\underbrace{\prod_{i=1}^{s}(\lambda_i^n-1)}_{U_n}
\underbrace{\prod_{i=s+1}^{s+2t}(\lambda_i^n-1)}_{V_n}
\underbrace{\prod_{i=2t+s+1}^{d}(\lambda_i^n-1)}_{W_n},
\end{equation}
where~$U_n$ is equal to the sum of~$\Lambda^n$ and~$(2^s-1)$ terms
comprising products of eigenvalues, each no larger
than~$\kappa^{-n}\vert\Lambda\vert^n$ in modulus,~$W_n$ is equal to
the sum of~$(-1)^{d-s}$ and~$2^{d-2t-s}-1$ terms bounded above in
absolute value by~$\kappa^{-n}$, and~$\vert V_n\vert\le2^{2t}$. It
follows that
\begin{eqnarray*}
\frac{\left\vert\prod_{i=1}^{d}(\lambda_i^n-1)-
(-1)^{d-s}\Lambda^n\prod_{i=s+1}^{s+2t}(\lambda_i^n-1)
\right\vert}{\vert\Lambda\vert^{n}}&=& \frac{\left\vert V_n\(U_nW_n-
(-1)^{d-s}\Lambda^n\)\right\vert}{\vert\Lambda\vert^{n}}\\
&=&\frac{\left\vert
V_n\(\Lambda^n+\bigo\(\Lambda^n/\kappa^n\)-\Lambda^n\)
\right\vert}{\vert\Lambda\vert^n}\\
&=&\bigo(\kappa^{-n}).
\end{eqnarray*}
The statement of the lemma follows by the reverse triangle
inequality.
\end{proof}

Now
\[
M_T(N)=\sum_{n=1}^{N}\frac{1}{n\vert\Lambda\vert^n}\(F_T(n)+\sum_{d\divides
n,d<n}\mu\(\textstyle\frac{n}{d}\)F_T(d)\)
\]
and
\[
\left\vert
\sum_{n=N}^{\infty}\frac{1}{n\vert\Lambda\vert^n}\sum_{d\divides n,
d<n}\mu\(\textstyle\frac{n}{d}\)F_T(d)
\right\vert\le
\sum_{n=N}^{\infty}\frac{1}{n}\cdot n\cdot\bigo(\vert\Lambda\vert^{-n/2})
=\bigo\(\vert\Lambda\vert^{-N/2}\),
\]
so there is a\mc{constantwithmoebiushalfn}
constant~$C_{\ref{constantwithmoebiushalfn}}$ for which
\begin{equation*}
\left\vert
\sum_{n=1}^{N}\frac{1}{n\vert\Lambda\vert^n}\sum_{d\divides n,
d<n}\mu\(\textstyle\frac{n}{d}\)F_T(d)-
C_{\ref{constantwithmoebiushalfn}}\right\vert=
\bigo\(\vert\Lambda\vert^{-N/2}\).
\end{equation*}
Therefore, by Lemma~\ref{lemma:1} and using the notation
from~\eqref{equation:definesABC},
\begin{equation*}
M_T(N)=\sum_{n=1}^{N}\frac{1}{n}\(V_n+\bigo\(\kappa^{-n}\)\)+C_{\ref{constantwithmoebiushalfn}}
+\bigo\(\vert\Lambda\vert^{-N/2}\).\label{equation:basicrelationship}
\end{equation*}
Clearly\mc{constantforsumofbigoepsilonminusn} there is a
constant~$C_{\ref{constantforsumofbigoepsilonminusn}}$ for which
\begin{equation}\label{equation:bigoepsilonbound}
\left\vert\sum_{n=1}^{N}\frac{1}{n}\bigo\(\kappa^{-n}\)-C_{\ref{constantforsumofbigoepsilonminusn}}\right\vert
=\bigo\(\kappa^{-N}\),
\end{equation}
so by~\eqref{equation:basicrelationship}
and~\eqref{equation:bigoepsilonbound},
\begin{equation}\label{equation:mainboundforMTN}
M_T(N)=\sum_{n=1}^{N}\frac{1}{n}V_n+C_{\ref{constantwithmoebiushalfn}}+
C_{\ref{constantforsumofbigoepsilonminusn}}+\bigo(R^{-N})
\end{equation}
where~$R=\min\{\kappa,\vert\Lambda\vert^{1/2}\}.$
Since the
complex eigenvalues appear in conjugate pairs we may arrange
that~$\lambda_{i+t}=\bar{\lambda_{i}}$ for~$s+1\le i\le s+t$, and
then
\[
\vert\lambda_i-1\vert\vert\lambda_{i+t}-1\vert=(\lambda_i-1)(\lambda_{i+t}-1).
\]
It follows that~$V_n=\prod_{i=s+1}^{s+2t}(\lambda_i^n-1)$.
Put
\[
\Omega=\left\{\prod_{i\in I}\lambda_i \mid
I\subseteq\{s+1,\ldots,s+2t\}\right\},
\]
write
\[
{\mathcal I}(\omega)=\{I\subset\{s+1,\dots,s+2t\}\mid\prod_{i\in I}\lambda_i=\omega\},
\]
\[
K(\omega)=\sum_{I\in\mathcal{I}(\omega)}(-1)^{\vert I\vert},
\]
and let~$m=K(1)$ (notice that~$\mathcal{I}(\omega)=\emptyset$
unless~$\omega\in\Omega$).
Then~$V_n=\sum_{\omega\in\Omega}K(\omega)\omega^n$ so,
by~\eqref{equation:mainboundforMTN},\mc{anotherconstant}\mc{thatconstantplusgamma}
\begin{eqnarray*}
M_T(N)&=&\sum_{n=1}^{N}\frac{1}{n}\sum_{\omega\in\Omega}K(\omega)
\omega^n+C_{\ref{anotherconstant}}+\bigo\(R^{-N}\)\\
&=&m\sum_{n=1}^{N}\frac{1}{n}+\sum_{\omega
\in\Omega\setminus\{1\}}K(\omega)\sum_{n=1}^{N}\frac{\omega^n}{n}
+C_{\ref{anotherconstant}}+\bigo\(R^{-N}\)\\
&=& m\log N-\sum_{\omega\in\Omega\setminus\{1\}}K(\omega)
\log(1-\omega)+C_{\ref{thatconstantplusgamma}}+\bigo(N^{-1}),
\end{eqnarray*}
since~$\sum_{n=1}^{N}\frac{1}{n}=\log N+\gamma+\bigo(N^{-1})$,
and
$\sum_{n=1}^{N}\frac{\omega^n}{n}=-\log(1-\omega)+\bigo(N^{-1})
$
for~$\omega\neq1$ by the Abel continuity theorem and partial
summation.

If the eigenvalues of modulus one
are~$\eul^{\pm2\pi{\imag}\theta_1},\dots,
\eul^{\pm2\pi{\imag}\theta_{t}}$ then
\[
V_n=\prod_{i=1}^{t}(1-\eul^{2\pi{\imag}\theta_in})
(1-\eul^{-2\pi{\imag}\theta_in})
=\prod_{i=1}^{t}\(2-2\cos(2\pi\theta_i n)\).
\]
Let~$X\subset\mathbb T^{t}$ be
the closure
of~$\{(n\theta_1,\dots,n\theta_{t})\mid n\in\mathbb Z\}$,
so that by the Kronecker--Weyl lemma we have
\[
\frac{1}{N}\sum_{n=1}^{N}\prod_{i=1}^{t}\(2-2\cos(2\pi\theta_i n)\)
\longrightarrow \int_{X}\prod_{i=1}^{t}\(2-2\cos(2\pi x_i)\)\dee
x_1\dots\dee x_t
\]
as~$N\to\infty$. Then, by partial summation,
\begin{eqnarray*}
\sum_{n=1}^{N}\frac{1}{n}V_n&=&\sum_{n=1}^{N}\(\frac{1}{n}-
\frac{1}{n+1}\)\sum_{m=1}^{n}V_m+\frac{1}{N+1}\sum_{m=1}^{N}V_m\\
&\sim&
\(\int_{X}\prod_{i=1}^{t}\(2-2\cos(2\pi x_i)\)\dee x_1\dots\dee x_t\)\log N,
\end{eqnarray*}
so that~$m$ has the form stated.
\end{proof}

The exact value of~$m$ is determined by the structure of the
group~$X$, which in turn is governed by additive relations among the
arguments of the eigenvalues of unit modulus. Here are some
illustrative examples.

\begin{example}\label{examples}
\noindent(a) If all the arguments~$\theta_i$ are independent
over~$\mathbb Q$ (the generic case), then~$X=\mathbb T^t$, so
\begin{equation*}
m\negthinspace=\negthinspace\int_0^1
\negthinspace\cdots\negthinspace\int_0^1\prod_{i=1}^{t}
\(2\negthinspace-\negthinspace 2\cos(2\pi x_i)\)\dee x_1\dots
\dee x_{t}
\negthinspace=\negthinspace
\( \int_0^1(2\negthinspace-\negthinspace 2\cos(2\pi x_1))\dee x_1
\negthinspace\negthinspace\)^{\negthinspace t}\negthinspace\negthinspace=2^t.
\end{equation*}

\smallskip
\noindent(b) A simple example with~$m>2^t$ is the following.
Let~$T_2$ be the automorphism of~$\mathbb T^8$ defined by the
matrix~$A\oplus A$, where
\begin{equation}\label{definesA}
A=\begin{pmatrix}0&0&0&-1\\1&0&0&8\\0&1&0&-6\\0&0&1&8
\end{pmatrix}.
\end{equation}
Here~$X$ is a diagonally embedded circle, and
\begin{eqnarray*}
m&=&\iint_{\{x_1=x_2\}}\prod_{j=1}^{2}
\(2-2\cos(2\pi jx_j)\)\dee x_1\dee x_2\\
&=&\int_0^1(2-2\cos(2\pi x))^2\dee x=6>2^2.
\end{eqnarray*}
Extending this example, let~$T_n$ be the automorphism
of~$\mathbb T^{4n}$ defined by the matrix~$A\oplus\cdots\oplus
A$ ($n$ terms). The matrix corresponding to~$T_n$ has~$2n$
eigenvalues with modulus one (comprising two conjugate
eigenvalues with multiplicity~$n$). Then~$X$ is again a
diagonally embedded circle, and
\begin{eqnarray*}
m
&=&\int_0^1(2-2\cos(2\pi x))^t\dee x\ =\  \frac{(2t)!}{(t!)^2}\ \sim\
\frac{2^{2t}}{\sqrt{\pi t}}
\end{eqnarray*}
by Stirling's formula.
This is much larger than~$2^t$,
reflecting the density of the syndetic set on
which the almost-periodic factor is
close to~$2^{2t}$. Indeed,
this example shows that~$\frac{m}{2^t}$ may be arbitrarily large.

\smallskip
\noindent(c) A simple example with~$m< 2^t$ is the following.
Let~$S$ be the automorphism of~$\mathbb T^{12}$ defined by the
matrix~$A\oplus A^2\oplus A^3$, with~$A$ as
in~\eqref{definesA}. Again~$X$ is a diagonally embedded circle,
and
\begin{eqnarray*}
m&=&\iiint_{\{x_1=x_2=x_3\}}\prod_{j=1}^{3}
\(2-2\cos(2\pi jx_j)\)\dee x_1{\rm d}x_2\dee x_3\\
&=&\int_0^1(2-2\cos(2\pi x))(2-2\cos(4\pi x))(2-2\cos(6\pi
x))\dee x\ =\ 6\ <\ 2^3.
\end{eqnarray*}
Extending this example, the value of~$m$ for the automorphism
of~$\mathbb T^{4t}$ defined by the matrix~$A\oplus
A^2\oplus\cdots\oplus A^t$ as~$t$ varies gives the sequence
\[ 2,4,6,10,12,20,24,34,44,64,78,116,148,208,286,410,556,808,1120,
1620,\dots
\]
(we thank Paul Hammerton for computing these numbers). This
sequence, entry~A133871 in the Encyclopedia of Integer
Sequences~\cite{MR95b:05001}, does not seem to be readily
related to other combinatorial sequences.

\smallskip
\noindent(d) Generalizing the example in~(c), for any
sequence~$(a_n)$ of natural numbers, we could look at the
automorphisms~$S_n$ of~$\mathbb T^{4n}$ defined by the
matrices~$\bigoplus_{k=1}^n A^{a_k}$, with~$A$ as
in~\eqref{definesA}. In order to make~$m$ small, we need a ``sum-heavy'' sequence, that is, one with many
three-term linear relations of the form~$a_i+a_j=a_k$.
More precisely, one would like many
linear relations with an odd number of terms, and few with an
even number of terms. Constructing such sequences, and understanding
how dense they may be, seems to be difficult.

Taking~$(a_n)$ to be the sequence whose first eight terms
are~$1,2,3,5,7,8,11,13$ and whose subsequent terms are defined
by the recurrence~$a_{n+8}=100a_n$, we find that the
automorphism~$S_{8n}$ of~$\mathbb T^{32n}$
has~$m=2^{4n}=2^{t/2}$. Thus~$\frac m{2^t}$ may be arbitrarily
small.
\end{example}

We close with some remarks.

\smallskip
\noindent(a) In the quasihyperbolic case the~$\bigo(1/N)$ term
is oscillatory, so no improvement of the asymptotic in terms of
a monotonic function is possible. The extent to which the
exponential dominance of the entropy term fails in this setting
is revealed by the following. Let~$F_T(n)$ denote the number of
points fixed by the automorphism~$T^n$. On the one hand,
Baker's theorem implies that~$F_T(n)^{1/n}\rightarrow e^h$
as~$n\to\infty$. On the other hand Dirichlet's theorem shows
that~$F_T(n+1)/F_T(n)$ does not converge
(see~\cite[Th.~6.3]{MR1461206}).

\smallskip
\noindent(b) The formula for~$m$ in the statement
of~\cite[Th.~1]{MR1787646} is incorrect in a minor way; as stated
in~\cite[Rem.~2]{MR1787646} and as illustrated in the examples
above,~$m$ should be~$K(1)$, which is not necessarily the same
as~$2^{t}$.

\smallskip
\noindent(c) The proof of Theorem~\ref{theorem} also gives an
elementary proof of the asymptotics in the hyperbolic case: in
the notation of the proof, $V_n=1$ so $m=1$. Applying now the
Euler-MacLaurin summation formula (see Ram Murty~\cite[Th.
2.1.9]{MR1803093}) we get an asymptotic of the shape
\[
\sum_{\vert\tau\vert\le N}\frac{1}{\eul^{h\vert\tau\vert}}=\log
N+C_{\ref{mainconstantquasicase}}+\sum_{r=0}^{k-1}\frac{B_{r+1}}{(r+1)N^{r+1}}+\bigo\(N^{-(k+1)}\),
\]
where $B_1=-\frac 12$, $B_2=\frac 16$,\ldots are the Bernoulli
numbers, for any~$k\ge1$.



\begin{thebibliography}{10}

\bibitem{MR1461206}
V.~Chothi, G.~Everest, and T.~Ward.
\newblock {$S$}-integer dynamical systems: periodic points.
\newblock {\em J. Reine Angew. Math.} \textbf{489}, 99--132, 1997.


\bibitem{MR2339472}
G.~Everest, R.~Miles, S.~Stevens, and T.~Ward.
\newblock Orbit-counting in non-hyperbolic dynamical systems.
\newblock {\em J. Reine Angew. Math.} \textbf{608}, 155--182, 2007.

\bibitem{MR1700272}
G. Everest and T. Ward.
\newblock{\em Heights of polynomials and entropy in algebraic dynamics.}
\newblock{Springer-Verlag London} (1999).




\bibitem{MR1803093}
M.~Ram Murty.
\newblock {\em Problems in analytic number theory}, volume 206 of {\em Graduate
  Texts in Mathematics}.
\newblock Springer-Verlag, New York, 2001.

\bibitem{MR1787646}
Mohd. Salmi~Md. Noorani.
\newblock Mertens theorem and closed orbits of ergodic toral automorphisms.
\newblock {\em Bull. Malaysian Math. Soc. (2)} \textbf{22}(2), 127--133, 1999.

\bibitem{apisit}
A. Pakapongpun and T. Ward.
\newblock{Functorial orbit counting}
\newblock{\em J. Integer Sequences},
\textbf{12}, Article 09.2.4, 2009.



\bibitem{MR1139566}
R.~Sharp.
\newblock An analogue of {M}ertens' theorem for closed orbits of {A}xiom {A}
  flows.
\newblock {\em Bol. Soc. Brasil. Mat. (N.S.)}, \textbf{21}(2), 205--229, 1991.

\bibitem{MR95b:05001}
N.~J.~A. Sloane.
\newblock An on-line version of the encyclopedia of integer sequences.
\newblock {\em Electron. J. Combin.} \textbf{1}:Feature 1, approx.\ 5 pp., 1994.
\newblock {\tt www.research.att.com/\~{ }njas/sequences/}.


\end{thebibliography}
\def\cprime{$'$}

\end{document}